\documentclass[bj]{amsart}
\usepackage{amssymb,mathrsfs,bm,enumerate}
\usepackage{amsfonts,dsfont,
  amsaddr,color}
\usepackage{verbatim}

\usepackage{txfonts}

\AtBeginDocument{\mathcode`v=\varv}
\AtBeginDocument{\mathcode`w=\varw}

\makeatletter
\newtheorem{theorem}{Theorem}[section]
\newtheorem{lemma}[theorem]{Lemma}
\newtheorem{corollary}[theorem]{Corollary}

\theoremstyle{definition}
\newtheorem{remark}[theorem]{Remark}

\allowdisplaybreaks \numberwithin{equation}{section}

\newbox\ovlbox
\def\ovl#1{\setbox\ovlbox\hbox{$#1$}\rlap{\kern.5\wd\ovlbox\kern-1.5pt
    $\overline{\hbox to4pt{\hss$\phantom{#1}$\hss}}$\hss}#1}
\def\unl#1{\setbox\ovlbox\hbox{$#1$}\rlap{\kern.5\wd\ovlbox\kern-2.5pt
    $\underline{\hbox to4pt{\hss$\phantom{#1}$\hss}}$\hss}#1}

\newcommand{\eps}{\varepsilon} \newcommand\1{\hbox{\kern.375em\vrule
    height1.57ex depth-.1ex width.05em\kern-.375em \rm 1}}
\newcommand\mequal{\overset{\text{\tiny m}}{=}}

 \newcommand\E{\mathbb{E}}
 
 \newcommand\R{\mathbb{R}}

 \def\tr{{\operatorname{tr}}\kern 0.08em}

\def\Ric{{\operatorname{Ric}}} \def\Hess{{\operatorname{Hess}}}

\def\mathpal#1{\mathop{\mathchoice{\text{\rm #1}}%
    {\text{\rm #1}}{\text{\rm #1}}%
    {\text{\rm #1}}}\nolimits} \def\id{{\mathpal{id}}}

\newcommand{\ptr}{/\!/}

\def\bd{{\bf d}} \def\r{\right} \def\l{\left} \def\e{\operatorname{e}}

\makeatother

\begin{document}

\title[Bismut-Stroock Hessian formulas] {Bismut-Stroock Hessian
  formulas and local Hessian estimates for heat semigroups and
  harmonic functions on Riemannian manifolds}

\author{Qin-Qian Chen\textsuperscript{1},\quad Li-Juan
  Cheng\textsuperscript{1},\quad Anton Thalmaier\textsuperscript{2}}

\address{\textsuperscript{1}Department of Applied Mathematics, Zhejiang University of Technology\\
  Hangzhou 310023, The People's Republic of China}
\address{\textsuperscript{2} Department of Mathematics, University of
  Luxembourg, Maison du Nombre,\\
  L-4364 Esch-sur-Alzette, Luxembourg}
 
\email{chenqq@mail.zjut.edu.cn} \email{chenglj@zjut.edu.cn}
\email{anton.thalmaier@uni.lu}

\begin{abstract}
  In this article, we develop a martingale approach to localized
  Bismut-type Hessian formulas for heat semigroups on Riemannian
  manifolds. Our approach extends the Hessian formulas established by
  Stroock (1996) and removes in particular the compact manifold
  restriction. To demonstrate the potential of these formulas, we give
  as application explicit quantitative local estimates for the Hessian
  of the heat semigroup, as well as for harmonic functions on regular
  domains in Riemannian manifolds.
\end{abstract}

 \date{\today}

\maketitle

\section{Introduction\label{s1}}

Let $M$ be a complete Riemannian manifold of dimension $n$ with
Levi-Civita connection $\nabla$. We consider the operator
$L=\frac{1}{2}\Delta$ where $\Delta:=\tr\,\nabla \bd $ is the
Laplace-Beltrami operator on $M$.  Denote by $X{.}(x)$ a Brownian
motion on $M$ starting at $x\in M$ with generator $L$ and explosion
time $\zeta(x)$. The explosion time is the random time at which the
process leaves all compact subsets of $M$. Furthermore, let $B_t$ be
the stochastic anti-development of $X{.}(x)$ which is a Brownian
motion on $T_xM$, stopped at the lifetime $\zeta(x)$ of $X{.}(x)$, and
denote by $P_tf$ the associated minimal heat semigroup, acting on
bounded measurable functions $f$, which is represented in
probabilistic terms by the formula
\begin{align*}
  (P_tf)(x)=\E\big[f(X_t(x))\1_{\{t< \zeta(x)\}}\big].
\end{align*}
Recall that for any $T>0$ fixed, $(t,x)\mapsto (P_tf)(x)$ is smooth
and bounded on ${(0,T]}\times M$ with $P_0f=f$.

Let $Q_t\colon T_xM\rightarrow T_{X_t}M$ be defined by
\begin{align*}
  DQ_t=-\frac{1}{2}\Ric^{\sharp}(Q_t)\,dt, \quad  Q_0=\id,
\end{align*}
where $D:=\ptr_t\,d \,\ptr_t^{-1}$ and
$\ptr_t:=\ptr_{0,t}: T_{x}M \rightarrow T_{X_{t}(x)}M$ denotes
parallel transport along $X(x)$.  The following formula is taken from
\cite{Thalmaier97} and gives a typical probabilistic derivative
formula for the semigroup $P_tf$; it extends earlier formulas of
Elworthy-Li \cite{EL94}.  Results in this direction are centered
around Bismut's integration by parts formula \cite{Bismut}. To be
precise, let $D$ be a fixed regular domain in $M$ and $\tau_D$ be the
first exit time of $X{.}(x)$ from $D$ when started at $x\in D$.  A
regular domain is by definition an open, relatively compact connected
domain in $M$ with non-empty boundary.  Then for $v\in T_xM$,
\begin{align*}
  \nabla P_tf(x)=\E\left[f(X_t(x)) \1_{\{t<\zeta(x)\}} \int_0^{\tau_D \wedge t} \langle Q_t(\dot{k}(s)v), \ptr_s \,d  B_s \rangle\right],
\end{align*}
where the function $k\in L^{1,2}(\R^+, [0,1])$ satisfies the property
that $k(0)=0$ and $k(s)=1$ for all $s\geq \tau_D\wedge t$.  An
explicit choice of the random function $k$ in the above Bismut formula
can be used to derive sharp gradient estimates in various situations,
e.g. for heat semigroups, as well as for harmonic functions on regular
domains in Riemannian manifolds, see \cite{TW98, Wbook1} for explicit
results.  Recently, such formulas have also been applied to
quantitative estimates for the derivative of a $C^2$ function $u$ in
terms of local bounds on $u$ and $\Delta u$ \cite{CTT18}.

A similar approach has been used by Arnaudon, Driver and Thalmaier
\cite{ADT2007} to approach Cheng-Yau type inequalities for the
harmonic functions.  Actually, the effect of curvature on the
behavior of harmonic functions on a Riemannian manifold $M$ is a
classical problem (see e.g. \cite{Yau75,Schoen96}), see also
\cite{CJKS-2020} for characterizations of gradient estimates in the
context of metric measure spaces.  In this paper, we aim to clarify in
particular the effect of curvature on the Hessian of harmonic
functions on regular domains in Riemannian manifolds.  To this end, we
establish a local version of a Bismut type formula for $\Hess P_tf$
which is consistent with Stroock's formula in \cite{Stroock} while
here we follow a martingale approach and allow the manifold to be
non-compact.

We start by giving some background on Bismut type formulas for second
order derivatives of heat semigroups. A first such formula appeared in
Elworthy and Li \cite{EL94,Li} for a non-compact manifold, however
under strong curvature restrictions.  An intrinsic formula for
$\Hess P_tf$ was given by Stroock \cite{Stroock} for a compact
Riemannian manifold, while a localized intrinsic formula was obtained
by Arnaudon, Planck and Thalmaier \cite{APT} adopting a martingale
approach.  A localized version of the Hessian formula (still with
doubly stochastic damped parallel translations) for the Feynman-Kac
semigroup was derived by Thompson \cite{Thompson2019}.  The study of
the Hessian of a Feynman-Kac semigroup generated by a Schr\"odinger
operator of the type $\Delta+V$ where $V$ is a potential, has been
pushed forward by Li \cite{Li2016,Li2018}.  Very recently, Bismut-type
Hessian formulas have been used for new applications, see for instance
Cao-Cheng-Thalmaier \cite{CCT} where $L^p$ Calder\'{o}n-Zygmund
inequalities on Riemannian manifolds have been established under
natural geometric assumptions for indices~$p>1$.

The following formula is taken from \cite{APT} and gives a localized
intrinsic Bismut type formula for the Hessian of the heat semigroup.
For $v, w\in T_xM$ define an operator-valued process
$W_t(v, w):T_xM\rightarrow T_{X_t}M$ as solution to the covariant
It\^{o} equation
\begin{align*}D W_t(v, w)=R\big(\ptr_td B_t,
  Q_t(v)\big)Q_t(w)-\frac{1}{2}(\bd^*R+\nabla \Ric)^{\sharp}(Q_t(v),
  Q_t(w))\,d t-\frac{1}{2}\Ric^{\sharp}(W_t(v, w))\,dt
\end{align*}
with initial condition $W_0(v,w)=0$. Here $R$ denotes the Riemann
curvature tensor.  It is easy to check that
$$W_t(v,w)=Q_t\int_0^t Q_r^{-1}\l(R\big(\ptr_r d B_r,
Q_r(v)\big)Q_r(w)-\frac{1}{2}(\bd^*R+\nabla \Ric)^{\sharp}(Q_r(v),
Q_r(w))\,d r\r).$$ The operator $\bd^*R$ is defined by
$\bd^*R(v_1)v_2:=-\tr\nabla{.}R(\cdot, v_1)v_2$ and satisfies
\begin{align*}
  \langle \bd^*R(v_1)v_2,v_3\rangle=\langle(\nabla_{v_3}\Ric^{\sharp})(v_1),v_2 \rangle-\langle (\nabla_{v_2}\Ric^{\sharp})(v_3),v_1 \rangle
\end{align*}
for all $v_1,v_2,v_3\in T_xM$ and $x\in M$.  Let $v,w\in T_xM$ with
$x\in D$, $f\in \mathcal{B}_b(M)$ and fix $0<S<T$.  Suppose that $D_1$
and $D_2$ are regular domains such that $x\in D_1$ and
$\bar{D}_1\subset D_2\subset D$. Let $\sigma$ and $\tau$ be two
stopping times satisfying
$0<\sigma\leq \tau_{D_1}<\tau\leq \tau_{D_2}$.  Assume $k,\ell$ are
bounded adapted processes with paths in the Cameron-Martin space
$L^{1,2}([0,T]; [0,1])$ such that
\begin{itemize}
\item $k(0)=1$ and $k(s)=0$ for $s\geq \sigma \wedge S$;
\item $\ell(s)=1$ for $s\leq \sigma \wedge S$ and $\ell(s)=0$ for
  $s\geq \tau \wedge T$.
\end{itemize}
Let $\Hess=\nabla \bd$.  Then for $f\in \mathcal{B}_b(M)$, we have
\begin{align}
  (\Hess_x P_{T}f)(v,v)=
  &-\E^x\l[f(X_T)\1_{\{T< \zeta(x)\}}\int_0^{T\wedge \tau}\langle W_s(\dot{k}(s)v,  v), \ptr_s d  B_s\rangle\r]\notag\\
  &+\E\left[f(X_T)\1_{\{T<\zeta(x)\}}\int_S^{T}\langle Q_s(\dot{\ell}(s)v), \ptr_s d  B_s\rangle\int_0^{S}\langle Q_s(\dot{k}(s)v), \ptr_s d  B_s\rangle\right].
    \label{Eq:HessianSemigr}
\end{align}
This formula has the advantage to be concise with only two terms on
the right-hand side to characterize the Hessian of semigroup; the
martingale approach to this formula is direct and does not need
advanced path perturbation theory on Riemannian manifold. However, it
depends on the choice of two random test functions $k$ and $\ell$.
For local estimates the construction of two different random times
$\tau$ and $\sigma$ is required, which complicates the calculation of
the coefficients in the estimates, see \cite{APT}.

Let us compare formula \eqref{Eq:HessianSemigr} with the Hessian
formula from \cite{Stroock} where the manifold is supposed to be
compact. Stroock proved that in this case one deterministic function
test function $k\in C^{1}([0,T], [0,1])$ satisfying $k(0)=1$ and
$k(t)=0$ is sufficient in a Bismut-type formula for the Hessian of the
semigroup (see \cite[Eq.~(1.6)]{Stroock97} and formula
\eqref{Hessian-formula-local} below). In what follows, we call it the
Bismut-Stroock Hessian formula.  Although its proof relies on
perturbation theory of Brownian paths and more terms are involved in
the formula, the fact that only one test function $k$ enters makes it
attractive for applications to the short-time behavior of the heat
kernel, see \cite{Stroock97}.  Recently, Chen, Li and Wu \cite{CLW}
adopted a perturbation of the driving force $B_t$ with a second order
term for a new approach to this formula.  The formula is derived with
localized vector fields and extends to a general (non-compact)
complete Riemannian manifold by introducing a family of cut-off
processes.

The described works motivate the following two questions:
\begin{enumerate}[(i)]
\item \emph{Can one extend the Stroock-Bismut Hessian formula
    \cite{Stroock} to a local Bismut-type formula for the Hessian heat
    semigroup by a direct construction of suitable martingales, and
    achieve from it explicit local estimates of\/ $\Hess P_tf$ by a
    proper choice of~$k$?}
\item \emph{Can the localized Bismut-Stroock Hessian formula be
    transformed into a formula for the Hessian of harmonic functions
    on regular domains in Riemannian manifolds from where quantitative
    estimates can be derived?}
\end{enumerate}

In the sequel we answer both questions positively.  Note that for the
first question, it has been explained in \cite[Section 4]{CLW} that it
seems difficult to adopt the perturbation theory of $M$-valued
Brownian motions in order to replace the non-random vector field on
path space by a random one; the time reversed field will not be
adapted anymore and the It\^o integral no longer be well defined. This
reason motivates our search for a different stochastic approach to the
problem in the form of a direct martingale argument.  Actually,
establishing localized Bismut-Stroock Hessian formula of semigroup
relies on the proper construction of martingales which turns out to be
the main difficulty for the first question. As an application of our
formula, we can give explicit local Hessian estimates for the heat
semigroup: for any regular domain $D\subset M$ and $x\in D$,
\begin{align}
  |\Hess P_tf|(x)
  \leq  \inf_{\delta>0}\l\{\l(\frac{t}{2}\sqrt{ K_1^2+\frac{K_2^2}{\delta }}+\frac{2}{ t}\r)\exp\l(t\,\Biggl(K_0^-+ \frac{\delta}{2}+\frac{\pi \sqrt{(n-1)K^-_0}}{2 \delta_x} + \frac{\pi^2 (n+3)}{4\delta_x^2}\Biggl)\r)\r\}\|f\|_{D},\label{Eq:HessPtf}
\end{align}
where $\delta_x=\rho(x,\partial D)$ is the Riemannian distance of $x$
to the boundary of $D$.  Compared with known estimates,
e.g.~\cite{APT}, the constants in \eqref{Eq:HessPtf} are explicit and
concise.  The estimate is not essentially more complicated, compared
to the local gradient estimate \cite{TW98}.

The second problem is then to extract Hessian estimates for harmonic
functions from the Bismut-Stroock Hessian formula which requires a
careful estimate of the terms involving $W{.}(v,w)$ and $Q{.}(v)$,
along with a proper choice of the process $k$. Let $D\subset M$ be a
regular domain.  In Section \ref{HF-section}, we obtain that for a
bounded positive harmonic function $u$ on $D$,
\begin{align*}
  |\Hess\, u|(x)&\leq \inf_{\substack{0<\delta_1<1/2\\ \delta_2>0}}\Bigg\{\sqrt{(1+2\1_{\{K_0^-\neq 0\}})\l((\delta_1+1)K_1^2+\frac{\delta_2}{2}K_2^2\r)}\notag\\
                &\quad +\sqrt{6}\,\Bigg(\frac{1}{2\delta_2}+3K_0^-+\frac{\pi \sqrt{(n-1)K^-_0}}{2 \delta_x} + \frac{\pi^2 (n+3+\delta_1^{-1})}{4\delta_x^2}\Bigg)\Bigg\}\sqrt{\|u\|_D\,u(x)},
\end{align*}
where constants $K_0$, $K_1$ and $K_2$ are defined as in
\eqref{Ric-D}--\eqref{dR-dRic} below; see also Theorem \ref{them-1}.

The paper is organized as follows. In Section \ref{section-HessPf} we
construct several local martingales to achieve local Hessian formulas
for the heat semigroup which are consistent with Stroock's
construction \cite{Stroock97} on compact manifolds with a
deterministic function $k$.  For reader's convenience, Section
\ref{HF-section} provides the method to construct~$k$ which follows
\cite{TW98}.  In Section 4 we give explicit local Hessian estimates
for harmonic functions. Finally, in Section 5, we apply our method
also to estimate the second derivative of heat semigroups and heat
kernels.

\section{Bismut-type formulas for the Hessian of  $P_tf$}\label{section-HessPf}
We start by constructing our fundamental martingale which will be the
basis for our approach.  To this end, for $v,w\in T_xM$  let
\begin{align*}
  W_{t}^k(v,w)=Q_t\int_0^tQ_{r}^{-1}R\big(\ptr_{r} d B_r, Q_{r}(k(r)v)\big)Q_{r}(w)-\frac12 Q_t\int_0^t Q_{r}^{-1}(\bd^* R+\nabla \Ric)^{\sharp}\big(Q_{r}(k(r)v), Q_{r}(w)\big)\, d r.
\end{align*}
It is easy to see that $W_{t}^k$ solves the following equation:
\begin{align*}
  DW^k_{t}(v, w)=R\big(\ptr_{t}\, d B_t, Q_{t}(k(t)v)\big)Q_{t}(w)-\frac{1}{2}(\bd^* R+\nabla \Ric)^{\sharp}\big(Q_{t}(k(t)v), Q_{t}(w)\big)\,d t-\frac{1}{2}\Ric^{\sharp}(W^k_{t}(v, w))\,d t,
\end{align*}
and $W_{0}^k(v,w)=0$.\goodbreak

\begin{theorem}\label{th1}
  Let $x\in M$ and $D$ be a relatively compact open domain in $M$ such
  that $x\in D$.  Let $\tau$ be a stopping time such that
  $0<\tau\leq\tau_D$ where $\tau_D$ denotes the first exit time of
  $X(x)$ with starting point $x\in D$.  Fix $T>0$ and suppose that $k$
  is a bounded, non-negative and adapted process with paths in the
  Cameron-Martin space $L^{1,2}([0,T];\R)$ such that $k(s)=0$ for
  $s\geq T\wedge \tau$, $k(0)=1$. Then for $f\in \mathcal{B}_b(M)$ and
  $v,w\in T_xM$,
  \begin{align*}
    &(\Hess  P_{T-t}f)\big(Q_t(k(t)v), Q_t(k(t) v)\big)+(\bd P_{T-t}f)
      (W_t^k(v,k(t) v)) \notag\\
    &\quad -2\bd P_{T-t}f(Q_t(k(t) v))\int_0^t\langle Q_s(\dot{k}(s)v),
      \ptr_sd  B_s \rangle
%      -\bd P_{T-t}f(Q_t(k(t) v))\int_0^t\langle Q_s(\dot{k}(s)v),
%      \ptr_sd  B_s \rangle
      \notag\\
    &\quad -P_{T-t}f(X_t)\int_0^t\langle W_s^k(v,\dot{k}(s) v),\ptr_sd  B_s \rangle \notag\\
    &\quad +P_{T-t}f(X_t)\l(\l(\int_0^t\langle Q_s(\dot{k}(s)v),
      \ptr_sd  B_s \rangle\r)^2-\int_0^t |Q_s(\dot{k}(s)v)|^2  \,ds\r)
  \end{align*}
  is a local martingale on $[0, \tau\wedge T)$.
\end{theorem}

\begin{proof}
  First of all, by an approximation argument we may assume that
  $f\in C^{\infty}(M)$ and is constant outside a compact set so that
  $|\bd f|$ and $\Delta f$ are bounded.

  Let
  \begin{align*}
    N_t(v,v)=\Hess  P_{T-t}f(Q_t(v), Q_t(v))+(\bd P_{T-t}f)(W_t(v,v)).
  \end{align*}
  Then $N_t(v,v)$ is a local martingale, see for instance the proof of
  \cite[Lemma 2.7]{Thompson2019} with potential $V\equiv 0$.
  Furthermore, define
$$N^k_t(v,v)=\Hess  P_{T-t}f(Q_t(k(t)v), Q_t(v))+(\bd P_{T-t}f)(W^k_t(v,v)).$$ 
According to the definition of $W^k_t(v,v)$, resp.~$W_t(v,v)$, and in
view of the fact that $N_t(v,v)$ is a local martingale, it is easy to
see that
\begin{align}\label{local-M1}
  N_t^k(v,v)&-\int_0^t(\Hess  P_{T-s}f)(Q_s(\dot{k}(s)
              v), Q_s(v))\,ds
\end{align}
is a local martingale as well.  Replacing in $N_t^k(v,v)$ the second
argument $v$ by $k(t)v$, we further see that also
\begin{align}\label{local-M2-1st}
  N_t^k(v,k(t)v)&-\int_0^t(\Hess  P_{T-s}f)(Q_s(\dot{k}(s)
                  v), Q_s(k(t)v))\,ds\notag\\
                &-\int_0^t\Hess  P_{T-s}f(Q_s(k(s)v), Q_s(\dot{k}(s)v))\,ds -\int_0^t(\bd P_{T-s}f)(W^k_s(v,\dot{k}(s)v))\,ds \notag \\
                &+ \int_0^t\int_0^s(\Hess  P_{T-r}f)(Q_r(\dot{k}(r)
                  v), Q_r(\dot{k}(s)v))\,dr\ ds
\end{align}
is a local martingale. Note that $N_t^k(v,k(t)v)=N_t^k(v,v)\,k(t)$.
Exchanging the order of integration in the
last term shows that
\begin{align}
  &N_t^k(v,k(t)v)-\int_0^t(\Hess  P_{T-s}f)(Q_s(\dot{k}(s)
    v), Q_s(k(t)v))\,ds\notag\\
  &\qquad-\int_0^t\Hess  P_{T-s}f(Q_s(k(s)v), Q_s(\dot{k}(s)v))\,ds -\int_0^t(\bd P_{T-s}f)(W^k_s(v,\dot{k}(s)v))\,ds \notag \\
  &\qquad+ \int_0^t(\Hess  P_{T-r}f)(Q_r(\dot{k}(r)
    v), Q_r((k(t)-k(r))v))\,dr\notag\\
  &=N_t^k(v,k(t)v) -\int_0^t(\bd P_{T-s}f)(W^k_s(v,\dot{k}(s)v))\,ds -2\int_0^t\Hess  P_{T-s}f(Q_s(k(s)v), Q_s(\dot{k}(s)v))\,ds\label{local-M2}
\end{align}
is a local martingale. Moreover, by the formula
\begin{align}\label{Pt-martingale}
  P_{T-t}f(X_t)=P_Tf(x)+\int_0^t\bd P_{T-s}f(\ptr_sd  B_s)
\end{align}
and integration by parts,
\begin{align}\label{local-M2-1}
  \int_0^t(\bd P_{T-s}f)(W_s^k(v,\dot{k}(s)v))\,ds-P_{T-t}f(X_t)
  \int_0^t\langle W_s^k(v,\dot{k}(s)v), \ptr_s\,d  B_s \rangle
\end{align}
is a local martingale. Similarly, from the formula
\begin{align*}
  \bd P_{T-t}f(Q_t(k(t)v))=\bd P_{T}f(v)+\int_0^t(\Hess  P_{T-s}f)
  (\ptr_sd  B_s, Q_s(k(s) v))+\int_0^t\bd P_{T-s}f(Q_s(\dot{k}(s)v))\,ds
\end{align*}
it follows that
\begin{align}\label{local-M3}
  &\int_0^t(\Hess  P_{T-s}f)(Q_s(\dot{k}(s)v), Q_s(k(s) v))
    \,ds-\bd P_{T-t}f(Q_t(k(t) v))\int_0^t\langle Q_s(\dot{k}(s)v),
    \ptr_sd  B_s\rangle \notag\\
  &\quad+\int_0^t\bd P_{T-s}f(Q_s(\dot{k}(s)v))\,ds \,\int_0^t\langle Q_s(\dot{k}(s)v)
    \ptr_sd  B_s\rangle
\end{align}
is also a local martingale. Concerning the last term in
\eqref{local-M3}, we note that
\begin{align*}
  &\int_0^t\bd P_{T-s}f(Q_s(\dot{k}(s)v))\,ds \,\int_0^t\langle Q_s(\dot{k}(s)v),\ptr_sd  B_s\rangle - \int_0^t \bd P_{T-s}f(Q_s(\dot{k}(s)v))\int_0^s\langle Q_r(\dot{k}(r)v),\ptr_rd  B_r\rangle\,ds
\end{align*}
is a local martingale. Combining this with \eqref{local-M3} we
conclude that
\begin{align}\label{local-M3-1}
  &\int_0^t(\Hess  P_{T-s}f)(Q_s(\dot{k}(s)v), Q_s(k(s) v))
    \,ds-\bd P_{T-t}f(Q_t(k(t) v))\int_0^t\langle Q_s(\dot{k}(s)v),
    \ptr_sd  B_s\rangle \notag\\
  &\quad+\int_0^t \bd P_{T-s}f(Q_s(\dot{k}(s)v))\int_0^s\langle Q_r(\dot{k}(r)v),\ptr_rd  B_r\rangle\,ds
\end{align}
is a local martingale.  Using the local martingales
\eqref{local-M2-1} and \eqref{local-M3-1} to replace the last two
terms in \eqref{local-M2}, we conclude that
\begin{align}\label{Martingale-1}
  &(\Hess  P_{T-t}f)(Q_t(k(t)v), Q_t(k(t) v))+(\bd P_{T-t}f)
    (W_t^k(v,k(t) v))  \notag\\
  &\quad -P_{T-t}f(X_t)\int_0^t\langle W_s^k(v,\dot{k}(s) v),\ptr_sd  B_s \rangle \notag\\
  &\quad -2\bd P_{T-t}f(Q_t(k(t) v))\int_0^t\langle Q_s(\dot{k}(s)v),
    \ptr_sd  B_s \rangle  \notag\\
  &\quad +2\int_0^t\bd P_{T-s}f(Q_s(\dot{k}(s)v)) \int_0^s\langle Q_r(\dot{k}(r)v),\ptr_rd  B_r\rangle \,ds 
\end{align}
is a local martingale as well.  On the other hand, by the product rule
for martingales, we have
\begin{align}\label{eqn}
  \l(\int_0^t\langle Q_s(\dot{k}(s)v),
    \ptr_sd  B_s \rangle\r)^2-\int_0^t |Q_s(\dot{k}(s)v)|^2\,ds = 2\int_0^t \l(\int_0^s\langle Q_r(\dot{k}(r)v),
    \ptr_rd  B_r \rangle\r)\langle Q_s(\dot{k}(s)v),
    \ptr_sd  B_s \rangle
\end{align}
which along with \eqref{Pt-martingale} implies that
\begin{align*}
  & P_{T-t}f(X_t)\l(\l(\int_0^t\langle Q_s(\dot{k}(s)v),
    \ptr_sd  B_s \rangle\r)^2-\int_0^t |Q_s(\dot{k}(s)v)|^2\,ds\r) \\
  &\quad-2\int_0^t\bd P_{T-s}f(Q_s(\dot{k}(s)v)) \int_0^s\langle Q_s(\dot{k}(s)v),\ptr_sd  B_s\rangle \,ds
\end{align*}
is a local martingale.  Applying this observation to
Eq.~\eqref{Martingale-1}, we finally see that
\begin{align*}
  &(\Hess  P_{T-t}f)\big(Q_t(k(t)v), Q_t(k(t) v)\big)\notag\\
  &\ +(\bd P_{T-t}f)
    (W_t^k(v,k(t) v))  \notag\\
  &\ -2\bd P_{T-t}f(Q_t(k(t) v))\int_0^t\langle Q_s(\dot{k}(s)v),
    \ptr_sd  B_s \rangle  \notag\\
  &\ -P_{T-t}f(X_t)\int_0^t\langle W_s^k(v,\dot{k}(s) v),\ptr_sd  B_s \rangle \notag\\
  &\ +P_{T-t}f(X_t)\l(\l(\int_0^t\langle Q_s(\dot{k}(s)v),
    \ptr_sd  B_s \rangle\r)^2-\int_0^t |Q_s(\dot{k}(s)v)|^2\,ds\r)
\end{align*}
is local martingale. This completes the proof.
\end{proof}

In order to formulate explicit Hessian estimates we need to introduce some geometric
bounds.  Let $D\subset M$ be a regular domain and $\tau_D$ be the first
exit time of $X(x)$ from $D$. We consider the following constants:
\begin{align}
&K_0:=\inf\l\{{\rm Ric}(v,v)\colon y\in D,\ v\in T_yM,\ |v|=1\r\};\label{Ric-D}\\
& K_1:=\sup\l\{|R|(y)\colon y\in D \r\};\label{R-D}\\
&K_2:=\sup\l\{|(\bd^*R+\nabla \Ric)^{\sharp}(v,v)|(y)\colon y\in D,\ v,w\in T_yM,\ |v|=|w|=1\r\}. \label{dR-dRic}
\end{align}
For $x\in M$ and $v,w\in T_xM$, we remark that
\begin{align*}
|R^{\sharp, \sharp}(v,v)|_{\text{\tiny HS}}(x)\leq |R|(x)\,|v|\,|v|
\end{align*}
where
$$|R^{\sharp, \sharp}(v,w)|_{\text{\tiny HS}}:=\sqrt{\sum_{i,j=1}^nR(e_i, v,w,e_j)^2}$$
and $\{e_i\}_{i=1}^n$ denotes an orthonormal base of $T_xM$.\smallskip 

 We shall need the following lemma.

 \begin{lemma}\label{th3}
Keep the assumptions and the definition of $k$ as in Theorem \ref{th1}, and   
assume that $k$ is bounded by~$1$. If $\int_0^T|\dot{k}(t)|^{2q}\,dt\in L^1$ for some $q\in (1,+\infty]$, then
\begin{align*}
  \E\left[\int_0^T|W_t^k(v,\dot{k}(t)v)|^2\,dt\right]\leq   \l(\frac{(2p-1) ^{p}K_1^{2p}}{2^2\delta^{p-1}}+\frac{K_2^{2p}}{2^2\delta^{2p-1} }\r)^{1/p}\e^{(2K_0^-+\delta)T}T^{2/p}\,
  \l(\E\l[\int_0^T|\dot{k}(t)|^{2q}\,dt\r]\r)^{1/q}<\infty
\end{align*}
for $\delta>0$ and $1/p+1/q=1$.
In particular, if $K_2\equiv 0$ and $k\in C^1([0,T],[0,1])$ a deterministic function, then 
\begin{align*}
\E\left[\int_0^T|W^k_t(v,\dot{k}(t)v)|^2\,dt\right]\leq  K_1^2T\e^{2K_0^-T}\,\int_0^T|\dot{k}(t)|^2\,dt<\infty.
\end{align*}
\end{lemma}

\begin{proof}
  From the definition of the operator $Q_s$ and the constant $K_0$, we
  know that $|Q_s|\leq \e^{-K_0s/2}$. According to It\^{o}'s
  formula and the definition of $ K_1,K_2$, we have for
  $v,w\in T_xD$ and $s<\tau\leq\tau_D$,
  \begin{align}
    d  |W_s^k(v, v)|^{2p}
    &= p\,(|W^k_s(v,v)|^2)^{p-1}\l[2 \l\langle R\big(\ptr_sd  B_s, Q_s(k(s)v)\big)Q_s(v), W^k_s(v,v)\r\rangle+ |R^{\sharp,\sharp}\big(Q_s(k(s)v),Q_s(v)\big)|_{\text{\tiny HS}}^2\,ds\r.\notag\\
    &\quad- \l.\l\langle (\bd^* R+\nabla \Ric^{\sharp})(Q_s(k(s)v), Q_s(v)), W^k_s(v,v) \r\rangle\,ds- \Ric(W^k_s(v,v),W^k_s(v,v))\,ds\r]\notag\\
    &\quad +2p(p-1)\,(|W^k_s(v,v)|^2)^{p-1} |R^{\sharp,\sharp}\big(Q_s(k(s)v),Q_s(v)\big)|_{\text{\tiny HS}}^2\,ds\notag\\
    &\leq  p\, (|W^k_s(v,v)|^2)^{p-1}\l[ 2\l\langle R\big(\ptr_sd  B_s, Q_s(k(s)v)\big)Q_s(v), W^k_s(v,v)\r\rangle\r.+\notag\\
    &\quad+ \l. K_2 \e^{-K_0s} |v|^2\, |W^k_s(v,v)|\,ds-K_0 \,|W^k_s(v,v)|^2\,ds\r]\notag\\
    &\quad +p(2p-1)\,K_1^2\e^{-2K_0s}|v|^4\,(|W^k_s(v,v)|^2)^{p-1}\,ds\notag\\
    &\leq 2p\,(|W^k_s(v,v)|^2)^{p-1}\l\langle R\big(\ptr_sd  B_s, Q_s(k(s)v)\big)Q_s(v), W^k_s(v,v)\r\rangle\notag\\
    &\quad +p(2p-1)\, K_1^2\e^{-2K_0s}(|W^k_s(v,v)|^2)^{p-1}|v|^4\,ds\notag\\
    &\quad+ K_2 p\,\e^{-2K_0s}|v|^2(|W^k_s(v,v)|)^{2p-1}\,ds-K_0p\, |W^k_s(v,v)|^{2p}\,ds. \label{est-Ws}
  \end{align}
Taking into account that for $\delta>0$,
\begin{align*}
p(2p-1) K_1^2\e^{-2K_0s}(|W^k_s(v,v)|^2)^{p-1}|v|^4&\leq \frac{\delta p}{2} |W^k_t(v,v)|^{2p}
+\frac{(2p-1)^p}{2\delta^{p-1}}\e^{-2pK_0s}K_1^{2p}|v|^{4p};\\
 K_2 p\e^{-2K_0s}|v|^2(|W^k_s(v,v)|)^{2p-1}& \leq \frac{\delta p}{2} |W^k_t(v,v)|^{2p}+\frac{1}{2\delta ^{2p-1}}\e^{-2pK_0s}K_2^{2p}|v|^{4p},
\end{align*}
  we thus have
  \begin{align*}
    d \e^{(K_0p-\delta p )s} |W^k_s(v, v)|^{2p}
      &=2p\e^{(K_0p-\delta p)s} |W^k_s(v, v)|^{2(p-1)} \l\langle R\big(\ptr_sd  B_s, Q_s(k(s)v)\big)Q_s(v), W^k_s(v,v)\r\rangle\\
      &\quad+\e^{-(K_0p+\delta p)s}\l(\frac{(2p-1)^pK_1^{2p}}{2\delta^{p-1}}+\frac{K_2^{2p}}{2\delta^{2p-1} }\r)|v|^{4p}\,ds.
  \end{align*}
  Integrating this inequality from $0$ to $t\wedge \tau$ and
  taking expectation yields
  \begin{align*}
    &\E\left[\e^{(K_0-\delta)p(t\wedge \tau)}|W^k_{t\wedge \tau}(v,v)|^{2p}\right]
      \leq  \l(\frac{(2p-1)^pK_1^{2p}}{2\delta^{p-1}}+\frac{K_2^{2p}}{2\delta^{2p-1} }\r)\E\left[\int_0^{t\wedge \tau}\e^{-(K_0+\delta)ps}|v|^{4p}\,ds\right].
  \end{align*}
 If
  $|v|=1$, we then arrive at
  \begin{align*}
    &\E\left[\e^{(K_0-\delta)p(t\wedge \tau)}|W^k_{t\wedge \tau}(v,v)|^{2p}\right]
      \leq \l(\frac{(2p-1)^pK_1^{2p}}{2\delta^{p-1}}+\frac{K_2^{2p}}{2\delta^{2p-1} }\r)\E\left[\int_0^{t\wedge \tau}\e^{-(K_0+\delta)ps}\,ds\right],
  \end{align*}
  from where we conclude that
  \begin{align}\label{Weq}
    \E\l[|W^k_{t}(v,v)|^{2p}\r]
    &\leq \l(\frac{(2p-1)^pK_1^{2p}}{2\delta^{p-1}}+\frac{K_2^{2p}}{2\delta^{2p-1} }\r)\e^{(K_0^-+\delta)pt}\left[\int_0^{t}\e^{-(K_0+\delta)ps}\,ds\right] \notag\\
    &\leq   \l(\frac{(2p-1)^pK_1^{2p}}{2\delta^{p-1}}+\frac{K_2^{2p}}{2\delta^{2p-1} }\r)\e^{\delta pt} \e^{K_0^-pt}\int_0^t\e^{K_0^- ps}\,ds.
  \end{align}
 Thus we have
  \begin{align*}
\E\left[\int_0^T|W^k_t(v,\dot{k}(t)v)|^2\,dt\right]&=\E\left[\int_0^T\dot{k}(t)^2|W^k_t(v,v)|^2\,dt\right]\\ &\leq  \E \l[\l(\int_0^T\dot{k}(t)^{2q}\,dt\r)^{1/q} \l(\int_0^T|W^k_t(v,v)|^{2p}\,dt\r)^{1/p}  \r]\\
& \leq \l(\E\l[\int_0^T\dot{k}(t)^{2q}\,dt\r]\r)^{1/q}  \l(\E\l[\int_0^T|W^k_t(v,v)|^{2p}\,dt\r]\r)^{1/p}\\
& \leq  \l(\E\l[\int_0^T\dot{k}(t)^{2q}\,dt\r]\r)^{1/q}  \l(\frac{(2p-1)^pK_1^{2p}}{2\delta^{p-1}}+\frac{K_2^{2p}}{2\delta^{2p-1} }\r)^{1/p}\e^{\delta T} \\
&\qquad \times \l(\int_0^T\e^{K_0^-pt}\int_0^t\e^{K_0^- ps}\,ds dt\r)^{1/p}\\
& \leq  \l(\E\l[\int_0^T\dot{k}(t)^{2q}\,dt\r]\r)^{1/q}  \l(\frac{(2p-1)^pK_1^{2p}}{2^2\delta^{p-1}}+\frac{K_2^{2p}}{2^2\delta^{2p-1} }\r)^{1/p}\e^{\delta T+2K_0^-T} T^{2/p} 
\end{align*}
  which completes the proof of first inequality.
  
  In particular, if $k\in C^1([0,T])$, then 
  \begin{align}\label{Weq2}
\E\left[\int_0^T|W^k_t(v,\dot{k}(t)v)|^2\,dt\right]&=\left[\int_0^T\dot{k}(t)^2\,\E(|W^k_t(v,v)|^2)\,dt\right].
\end{align}
Moreover, from the calculation in \eqref{est-Ws}  with $K_2\equiv 0$ we see that
\begin{align*}
 d  |W_s^k(v, v)|^{2}
    &= 2\l\langle R\big(\ptr_sd  B_s, Q_s(k(s)v)\big)Q_s(v), W^k_s(v,v)\r\rangle\notag\\
    &\quad + K_1^2\e^{-2K_0s}|v|^4\,ds -K_0 |W^k_s(v,v)|^{2}\,ds 
\end{align*}
which implies 
\begin{align}\label{Eq:EstWt}
    \E|W_{t}^k(v,v)|^{2}
    \leq K^2_1\e^{2K_0^-T}T.
  \end{align}
  Combining \eqref{Weq2} and \eqref{Eq:EstWt} we obtain
  \begin{align*}
  \E\left[\int_0^T|W^k_t(v,\dot{k}(t)v)|^2\,dt\right]&\leq K^2_1\e^{2K_0^-T}T\left[\int_0^T\dot{k}(t)^2\,dt\right]
  \end{align*}
  which gives the additional claim.
\end{proof}

By means of Theorem \ref{th1} and Lemma \ref{th3} we are now in position
to formulate the localized Bismut-Stroock Hessian
formula as follows.

\begin{theorem}\label{th4}
  Let $x\in M$ and $D$ be a relatively compact open domain such that $x\in D$.  Let
  $\tau$ be a stopping time such that $0<\tau\leq\tau_D$.  Suppose that
  $k$ is a bounded, non-negative and adapted process with paths in the
  Cameron-Martin space $L^{1, 2q}([0,T]; \R)$ such that $k(t)=0$ for
  $t\geq T\wedge \tau$, $k(0)=1$ and
  $\int_0^{T\wedge \tau}|\dot{k}(t)|^{2q}\,dt\in L^1$ for some constant $q>1$. Then for
  $f\in \mathcal{B}_b(M)$ and $v,w\in T_xM$, we have
  \begin{align}\label{Hessian-formula-local}
    (\Hess  P_{T}f)(v,v)=
    &-\E^x\l[f(X_T)\1_{\{T<\zeta(x)\}}\int_0^{T\wedge \tau}\Bigg\langle Q_s\int_0^sQ_{r}^{-1}R(\ptr_{r} d B_r, Q_{r}(k(r)v))Q_{r}(\dot{k}(s)v), \ptr_s d  B_s\Bigg\rangle\r]\notag\\
    &-\E^x\l[f(X_T)\1_{\{T<\zeta(x)\}}\int_0^{T\wedge \tau}\Bigg\langle Q_s\int_0^sQ_{r}^{-1}(\bd^* R+\nabla \Ric)^{\sharp}(Q_{r}(k(r)v), Q_{r}(\dot{k}(s)v))\, d r, \ptr_s d  B_s\Bigg\rangle\r]\notag\\
    &+\E\left[f(X_T)\1_{\{T< \zeta(x)\}}\l(\l(\int_0^{T\wedge \tau}\langle Q_s(\dot{k}(s)v), \ptr_s d  B_s\rangle\r)^2-\int_0^{T\wedge \tau} | Q_s(\dot{k}(s)v)|^2 \,ds\r)\right].
  \end{align}
\end{theorem}

\begin{proof}
  For $\eps>0$ small, let $k^{\eps}(t)$ be a bounded, non-negative and
  adapted process with path in the Cameron-Martin space
  $L^{1, 2q}([0,T]; \R)$ such that $k^{\eps}(t)=0$ for
  $t\geq (T-\eps)\wedge \tau$, $k^{\eps}(0)=1$. We know that
$$|Q_{t\wedge \tau}|\leq \e^{-\frac{1}{2}K_0(t\wedge \tau)}\quad \text{and}\quad \E\left[\int_0^T|W_t^{k^{\eps}}(v,\dot{k}^{\eps}(t)v)|^2\,dt\right]<\infty.$$
By the boundedness of $P_tf$ on $[0,T]\times D$ and the boundedness of
$|\bd P_tf|$ and $|\Hess P_tf|$ on $[\eps,T]\times D$ since $P{.}f$ is
smooth on $(0,T]\times M$, it follows first that the local martingale
in Theorem \ref{th1} is a true martingale for the chosen
$k^{\eps}$. By taking expectations, along with the strong Markov
property, we get the formula
\begin{align*}
  (\Hess  P_{T}f)(v,v)=
  &-\E^x\l[f(X_T)\1_{\{T<\zeta(x)\}}\int_0^{T\wedge \tau}\langle W^{k^{\eps}}_s(v,\dot{k}^{\eps}(s)v), \ptr_s d  B_s\rangle\r]\\
  &+\E^x\left[f(X_T)\1_{\{T<\zeta(x)\}}\l(\l(\int_0^{T\wedge \tau}\langle Q_s(\dot{k}^{\eps}(s)v), \ptr_s d  B_s\rangle\r)^2-\int_0^{T\wedge \tau}| Q_s(\dot{k}^{\eps}(s)v)|^2 \,ds\r)\right].
\end{align*}
Finally, through approximating the given $k$ by an appropriate
sequence of $k^{\eps}$ as above and letting $\eps\downarrow 0$, we
finish the proof.
\end{proof}

\begin{remark}
\begin{enumerate}[(i)]
\item In view of Eq.~\eqref{eqn}, formula
  \eqref{Hessian-formula-local} also can be written as
  \begin{align}\label{Hessian-formula-local-2}
    (\Hess  P_{T}f)(v,v)=
    &-\E^x\l[f(X_T)\1_{\{T<\zeta(x)\}}\int_0^{T\wedge \tau}\langle W^k_s(v,\dot{k}(s)v), \ptr_s d  B_s\rangle\r]\notag\\
    &+2\E^x\left[f(X_T)\1_{\{T<\zeta(x)\}}\int_0^{T\wedge \tau}\l( \int_0^{s}\langle Q_r(\dot{k}(r)v), \ptr_r d  B_r\rangle\r)\langle Q_s(\dot{k}(s)v), \ptr_s d  B_s\rangle\right].
  \end{align}
\item Since $ (\Hess P_{T}f)$ is a symmetric form, it is
  straightforward from Theorem \ref{th4} that
  \begin{align}\label{Hessian-formula-local-3}
    (\Hess  P_{T}f)(v,w)=
    &-\frac{1}{2}\E^x\l[f(X_T)\1_{\{T<\zeta(x)\}}\int_0^{T\wedge \tau}\langle W^k_s(v,\dot{k}(s)w), \ptr_s d  B_s\rangle\r]\notag\\
    &-\frac{1}{2}\E^x\l[f(X_T)\1_{\{T<\zeta(x)\}}\int_0^{T\wedge \tau}\langle W^k_s(w,\dot{k}(s)v), \ptr_s d  B_s\rangle\r] \notag \\
    &+\E^x\left[f(X_T)\1_{\{T<\zeta(x)\}}\l(\int_0^{T\wedge \tau}\langle Q_s(\dot{k}(s)v),
      \ptr_sd  B_s \rangle\r)\l(\int_0^{T\wedge \tau}\langle Q_s(\dot{k}(s)w),
      \ptr_sd  B_s \rangle\r)\right]\notag\\
    &-\E^x\left[f(X_T)\1_{\{T<\zeta(x)\}}\int_0^{T\wedge \tau} \langle  Q_s(\dot{k}(s)v), Q_s(\dot{k}(s)w) \rangle \,ds\right]
 \end{align}                              
 for $v,w\in T_xM$ and $x\in D$.
 \end{enumerate}                               
 \end{remark}

Denoting by $p_t(x,y)$ the transition density of the diffusion with
generator $L$, using Theorem \ref{th4}, we obtain the following
Bismut-type Hessian formula for the logarithmic density.

\begin{corollary}\label{cor-heat-kernel}
  We keep the assumptions of Theorem \ref{th4}. Then
  \begin{align*}
    &\l(\Hess_x \log p_T(x,y)\r)(v,v) \\
    &= -\E^x\l[\int_0^{T\wedge \zeta(x)}\langle W^k_s(v,\dot{k}(s)v), \ptr_s d  B_s\rangle\,\Big|\,X_T=y\r]+\E^x\left[\l(\int_0^{T\wedge \zeta(x)}\langle Q_s(\dot{k}(s)v), \ptr_s d  B_s\rangle\r)^2\,\Big|\,X_T=y\right]\\
    &\quad-\E^x\left[\int_0^{T\wedge \zeta(x)}|  Q_s(\dot{k}(s)v)|^2 \,ds\,\Big|\,X_T=y\right]-\l(\E^x\left[\int_0^{T\wedge \zeta(x)}\langle Q_s(\dot{k}(s)v), \ptr_s d  B_s\rangle\,\Big|\,X_T=y\right]\r)^2
  \end{align*}
\end{corollary}

\begin{remark}
  Corollary \ref{cor-heat-kernel} is well adapted to determine the short-time
  behavior of the Hessian of the heat kernel on a non-compact
  manifold (see the procedure in \cite[Section 5]{CLW}).
\end{remark}

\section{Hessian estimates for Harmonic functions}\label{HF-section}

We now adjust Theorem \ref{th4} to the case of harmonic functions on
regular domains. Recall that by a regular domain we mean a connected
relatively compact open subset of $M$ with non-empty boundary. For
$D\subset M$ a regular domain and $f\in C^2(\bar{D})$, let
\begin{align*}
  |\Hess f|(x)=\sup\big\{ |(\Hess f)(v,w)|\colon |v|\leq 1, |w|\leq 1, \ v,w \in T_xM\big\},\quad x\in D.
\end{align*}
Since $\Hess f$ is symmetric, it is easy to see that
\begin{align*}
  |\Hess f|(x)=\sup\big\{ |(\Hess f)(v,v)|\colon |v|\leq 1,  \ v \in T_xM\big\},\quad x\in D.
\end{align*}

\begin{theorem}\label{them-harmonic}
Let $D\subset M$ be a regular domain, $u\in C^2(\bar{D})$ be 
harmonic on $D$, further $x\in D$ and $v,w\in T_xM$. Then for any
bounded %$\mathscr{F}{\!.}(x)$-
adapted process $k$ with sample paths
in $L^{1,2q}(\mathbb{R}^+;\mathbb{R})$ for some $q>1$ such that 
\begin{align*}
\int_0^{\tau_D}|\dot{k}(s)|^{2q}\,ds\in L^{1},
\end{align*} 
 and the property that $k(0)=0$ and $k(s)=1$ for all 
$s\geq \tau_D$, the following formula holds:
\begin{align*}
&(\Hess\,u)_x (v,v)\\
  &=-\E^x\l[u(X_{\tau_D}(x))\int_0^{\tau_D}\langle W^k_s(v,\dot{k}(s)v), \ptr_s d  B_s\rangle\r]\\
&\quad+\E^x\left[u(X_{\tau_D}(x))\l\{\l(\int_0^{\tau_D}\langle Q_s(\dot{k}(s)v), \ptr_s d  B_s\rangle\r)^2-\int_0^{\tau_D}  |Q_s(\dot{k}(s)v)|^2\,ds\r\}\right].
\end{align*}

\begin{proof}
Let $(P_t^Du)(x):=\E[u(X_{t\wedge\tau_D}(x))]$. Since $u$ is harmonic, we have 
$P^D_tu=u$ for $t\geq 0$, and then analogously to Theorem  \ref{th4},
\begin{align*}
(\Hess&\,u)_x (v,v)=(\Hess \,P^D_t u)_x (v,v)\\
=&-\E^x\l[u(X_{t\wedge \tau_D}(x))\int_0^{t\wedge\tau_D}\langle W^k_s(v,\dot{k}(s)v), \ptr_s d  B_s\rangle\r]\\
&+\E^x\left[u(X_{t\wedge \tau_D}(x))\l(\int_0^{t\wedge\tau_D}\langle Q_s(\dot{k}(s)v), \ptr_s d  B_s\rangle\r)^2\right]\\
&-\E^x\left[u(X_{t\wedge \tau_D}(x))\int_0^{t\wedge \tau_D}|Q_s(\dot{k}(s)v)|^2 \,ds\right].
\end{align*}
By letting $t$ tend to infinity, we obtain the claim.
\end{proof}

\end{theorem}
\begin{remark}
Taking Eq.~\eqref{eqn} into account, the above formula can be rewritten as follows:
\begin{align}\label{Hessian-harmonic-function}
(\Hess\,u)_x (v,v)
 =&-\E^x\l[u(X_{\tau_D}(x))\int_0^{\tau_D}\langle W^k_s(v,\dot{k}(s)v), \ptr_s d  B_s\rangle\r]\notag\\
&+2\E^x\left[u(X_{\tau_D}(x))\int_0^{\tau_D} \l(\int_0^{s}\langle Q_r(\dot{k}(r)v), \ptr_r d  B_r\rangle\r) \langle Q_s(\dot{k}(s)v), \ptr_s d  B_s\rangle\right].
\end{align}
\end{remark}

\subsection{Construction of the random function $k$}
We are now going to briefly sketch the method from
\cite{TW98,Th-Wang:2011} to construct the function $k$. We
restrict ourselves to Brownian motion on $D$ with lifetime $\tau_D$
(exit time from~$D$).  Let $f$ be a continuous function defined on
$\bar D$ which is strictly positive on $D$.  For $x\in D$, we consider
the strictly increasing process
\begin{align}\label{random-time1}
  T(t)=\int_0^tf^{-2}(X_s(x))\,ds,\quad t\leq \tau_D,
\end{align}
and let
\begin{align}\label{random-time2}
  \tau(t)=\inf\{s\geq 0: T(s)\geq t\}.
\end{align}
Obviously $T(\tau(t))=t$ since $\tau_D<\infty$, and $\tau(T(t))=t$ for
$t\leq \tau_D$.  Note that since $X$ is a diffusion with generator $\frac{1}{2}\Delta$,
the time-changed diffusion
$X_t':=X_{\tau(t)}$ has generator $L'=\frac{1}{2}f^2\Delta$. In
particular, the lifetime $T(\tau_D)$ of $X'$ is infinite, see \cite{TW98}.

Fix $t>0$ and let 
\begin{align*}
h_0(s)=\int_0^sf^{-2}(X_r(x)){1}_{\{r<\tau(t)\}}\,dr,\quad s\geq 0.
\end{align*}
Then for $s\geq \tau(t)$,
\begin{align*}
h_0(s)=h_0(\tau(t))=\int_0^{\tau(t)}f^{-2}(X_r(x))\,dr=t.
\end{align*}
Next, let $h_1\in C^1([0,t],\mathbb{R})$ with $h_1(0)=0$ and
$h_1(t)=1$. We consider $k(s)=1-(h_1\circ h_0)(s)$. Then $k$ is an
adapted bounded process with $k(0)=1$ and $k(s)=0$ for
$s\geq \tau(t)$.

\begin{lemma}\label{lem:est}
Suppose $f \in C^2(\ovl{D})$ with $0<f\leq 1$ on $D$, $f(x)=1$ and $f\vert_{\partial D} = 0$. For $\delta_1, \delta_2>0$ and $q>1$ set
\begin{align}
c_1(f)&:=\sup_D\l\{\l((2\delta_2)^{-1}+K_0^-\r)f^2+(3+\delta_1^{-1})|\nabla f|^2-f\Delta f\r\};\label{def-c1} \\
c_2(f)&:=\sup_D\l\{K_0^-f^2+5|\nabla f|^2-f\Delta f\r\};\label{def-c2}\\
c_3(f)&:=\sup_D\l\{K_0^-f^2+3|\nabla f|^2-f\Delta f\r\};\label{def-c3}\\
\tilde{c}_q(f) &:= \sup_D \Big\lbrace (2q+1)|\nabla f|^2 - f \Delta f \Big\rbrace \label{def-c4}
\end{align}
where $K_0$ is defined as in \eqref{Ric-D}.  Then,
\begin{align*}
c_1(f)&\leq  \frac1{2\delta_2}+K_0^-+\frac{\pi \sqrt{(n-1)K^-_0}}{2 \delta_x} + \frac{\pi^2 (n+3+\delta_1^{-1})}{4\delta_x^2}:=c_1,\\
c_2(f)&\leq K_0^-+\frac{\pi \sqrt{(n-1)K^-_0}}{2 \delta_x} + \frac{\pi^2 (n+5)}{4\delta_x^2}:=c_2,\\
c_3(f)&\leq K_0^-+\frac{\pi \sqrt{(n-1)K^-_0}}{2 \delta_x} + \frac{\pi^2 (n+3)}{4\delta_x^2}:=c_3,\\
\tilde{c}_q(f)& \leq \frac{\pi \sqrt{(n-1)K^-_0}}{2 \delta_x} + \frac{\pi^2 (n+2q+1)}{4\delta_x^2}:=\tilde{c}_q,
\end{align*}
where $\delta_x:=\rho(x,\partial D)$ denotes the Riemannian distance of $x$
to the boundary $\partial D$.
\end{lemma}
\begin{proof}
Take $$f(y)=\sin{\l(\frac{\pi\rho(y,\partial D)}{2}\r)},\quad y\in B(x,\delta_x).$$
This choice of $f$ clearly satisfies the conditions and furthermore,
\begin{align}\label{esti-grad-f}
|\nabla f|\leq \frac{\pi}{2\delta_x}.
\end{align}
By the Laplacian comparison theorem, we have
\begin{align*}
-\Delta f\leq \frac{\pi}{2\delta_x}\sqrt{(n-1)K_0^-}+\frac{\pi^2 n}{4\delta_x^2}
\end{align*}
which together with \eqref{esti-grad-f}  gives the claimed estimates on $\tilde{c}_q(f),c_1(f), c_2(f)$ and $ c_3(f)$.
\end{proof}
The following lemma is taken from \cite[Lemma 2.2]{CTT18} or \cite{TW98} with a trivial modification.

\begin{lemma}\label{lem-1}
Suppose $f \in C^2(\ovl{D})$ with $f >0$ and $f \leq 1$ on $D$, $f(x)=1$, $f\vert_{\partial D} = 0$ and set
$\tilde{c}_q(f)$ as in \eqref{def-c4}
for $q\geq 1$. Then there exists a bounded adapted process $k$ with paths in the Cameron-Martin space
$L^{1,2q}([0,t];\mathbb{R})$ for some $q\in [1,\infty)$ such that $k(0)=1$, $k(s)=0$ for $s\geq t\wedge \tau_D$ and 
\begin{align*}
\E\left[\int_0^{t\wedge \tau_D}|\dot{k}(s)|^{2q}\,ds\right]\leq \frac{\tilde{c}_q(f)}{1-\e^{-\tilde{c}_q(f)t}}\leq \frac{\e^{\tilde{c}_qt}}{t},
\end{align*}
where $\tilde{c}_q$ is defined in Lemma \ref{lem:est}.
\end{lemma}

\subsection{Estimate for the Hessian of a harmonic function}

In this subsection we focus on explicit Hessian estimates for harmonic functions defined
on regular domains in Riemannian manifolds.

\begin{theorem}\label{them-1}
Let $D\subset M$ be a regular domain. For $x\in D$, let $\delta_x=\rho(x,\partial D)$ be the Riemannian distance of $x$ to the boundary of $D$.
For a bounded positive harmonic function $u$ on $D$, one has
\begin{align}
|\Hess\, u|(x)&\leq \inf_{\substack{0<\delta_1<1/2\\ \delta_2>0}}\l\{\sqrt{(1+2\1_{\{K_0^-\neq 0\}})\l((\delta_1+1)K_1^2+\frac{\delta_2}{2}K_2^2\r)},\r.\notag\\
&\quad +\sqrt{6}\,\Bigg(\frac{1}{2\delta_2}+3K_0^-+\frac{\pi \sqrt{(n-1)K^-_0}}{2 \delta_x} + \frac{\pi^2 (n+3+\delta_1^{-1})}{4\delta_x^2}\l.\Bigg)^{\mathstrut}\r\}\sqrt{\|u\|_D\,u(x)};\label{EHF-1}\\
&\notag \\
|\Hess\, u|(x)&\leq \inf_{\substack{\delta_1>0\\ \delta_2>0}}\l\{\sqrt{(1+2\1_{\{K_0^-\neq 0\}})\l((\delta_1+1)K_1^2+\frac{\delta_2}{2}K_2^2\r)}\notag\r.\\
&\quad +\Bigg(\frac{1}{\delta_2}+6K_0^-+\frac{\pi \sqrt{(n-1)K^-_0}}{ \delta_x} + \frac{\pi^2 (n+3+\delta_1^{-1})}{2\delta_x^2}\l.\Bigg)^{\mathstrut}\r\}\,\|u\|_D, \label{EHF-2}
\end{align}
where $K_0$, $K_1$ and $K_2$ are defined as in \eqref{Ric-D}--\eqref{dR-dRic}.
\end{theorem}

\begin{proof}[Proof of Theorem \ref{them-1} (Part I)]
  We fix $x\in D$ and denote by $B=B(x,\delta_x)$ the open ball about
  $x$ of radius~$\delta_x$. Then Theorem \ref{them-harmonic} applies
  with $B$ now playing the role of $D$.  First, by formula
  \eqref{Hessian-harmonic-function} and the Cauchy inequality, we get
\begin{align}\label{eq-Hess-u}
&\big|(\Hess\,u)_x(v,v)\big|=\lim_{t\rightarrow \infty}\big|(\Hess \, P^D_tu)_x (v,v)\big|\notag\\
&\leq \left|\E^x\l[u(X_{\tau_D}(x))\int_0^{\tau_D}\langle W^k_s(v,\dot{k}(s)v), \ptr_s d  B_s\rangle\r]\,\right|\notag\\
&\quad +2\left|\E^x\left[u(X_{\tau_D}(x))\int_0^{\tau_D}\l( \int_0^{s}\langle Q_r(\dot{k}(r)v), \ptr_r d  B_r\rangle\r)\langle Q_s(\dot{k}(s)v), \ptr_s d  B_s\rangle\right]\,\right|\notag\\
&\leq   \sqrt{\|u\|_D\,u(x)}\sqrt{\lim_{t\rightarrow \infty}\E^x\left[\int_0^{\tau(t)}|W^k_s(v,v)|^2\dot{k}(s)^2\,ds\right]}\notag\\
&\quad+2\sqrt{\|u\|_D\,u(x)}\sqrt{\lim_{t\rightarrow \infty}\E^x\l[\int_0^{\tau(t)} \l(\int_0^{s}\langle Q_r(\dot{k}(r)v), \ptr_r d  B_r\rangle\r)^2 | Q_s(\dot{k}(s)v)|^2 \,ds\r]}:=\text{I}+\text{II},
\end{align}
where the stopping time $\tau(t)$  is defined by \eqref{random-time2}. 
For the first term I, according to the relation between $\tau(t)$ and~$f$, we can estimate as
\begin{align}
\E^x&\left[\int_0^{\tau(t)}|W^k_s(v,v)|^2\dot{k}(s)^2\,ds\right]\leq \E^x\left[\int_0^{\tau(t)}|W^k_s(v,v)|^2\dot{k}(s)^2\,ds\right]\notag\\
&\leq \E^x\left[\int_0^{\tau(t)}|W^k_s(v,v)|^2\dot{h}_1(h_0(s))^2 f^{-4}(X_{s})\,ds\right]\notag\\
&=\E^x\left[\int_0^{t}|W^k_{\tau(s)}(v,v)|^2\dot{h}_1(s)^2 f^{-4}(X_{\tau(s)})\,d\tau(s)\right]\notag\\
&=\int_0^{t}\dot{h}_1(s)^2\E^x\big[|W^k_{\tau(s)}(v,v)|^2 f^{-2}(X_{\tau(s)})\big]\,ds. \label{W-esti}
\end{align}
For the second term II, we find 
\begin{align*}
&\E^x\l[\int_0^{\tau(t)} \l(\int_0^{s}\langle Q_r(\dot{k}(r)v), \ptr_r d  B_r\rangle\r)^2 | Q_s(\dot{k}(s)v)|^2 \,ds\r]\\
  &\quad\leq \E^x\l[\int_0^{t}\dot{h}_1(s)^2f^{-2}(X_{\tau(s)})\e^{-K_0\tau(s)}\l(\int_0^{\tau(s)}\langle Q_r(\dot{k}(r)v), \ptr_r d  B_r\rangle\r)^2 ds\r].
\end{align*}
Substituting these estimates back into \eqref{eq-Hess-u} yields
\begin{align}\label{eq1}
  \big|(\Hess \,u)_x(v,v)\big|
  &\leq \sqrt{\|u\|_D\,u(x)}\sqrt{\lim_{t\rightarrow \infty}\int_0^{t}\dot{h}_1(s)^2\,\E^x\big[|W^k_{\tau(s)}(v,v)|^2f^{-2}(X_{\tau(s)})\big]\,ds}\notag\\
  &\quad +2\sqrt{\|u\|_D\,u(x)}\sqrt{\lim_{t\rightarrow \infty}\int_0^{t}\dot{h}_1(s)^2\,\E^x\l[f^{-2}(X_{\tau(s)})\e^{-K_0\tau(s)}\l(\int_0^{\tau(s)}\langle Q_r(\dot{k}(r)v), \ptr_r d  B_r\rangle\r)^2\r]\,ds}.
\end{align}
Let
\begin{align*}
  &\Phi_1(s):=\E^x\big[|W^k_{\tau(s)}(v,v)|^2f^{-2}(X_{\tau(s)})\big]; \\
  &\Phi_2(s):=\E^x\l[f^{-2}(X_{\tau(s)})\e^{-K_0\tau(s)}\l(\int_0^{\tau(s)}\langle Q_r(\dot{k}(r)v), \ptr_r d  B_r\rangle\r)^2\r].
\end{align*}
It\^{o}'s formula gives
\begin{align*}
  d &\l(f^{-2}(X_{\tau(s)})|W^k_{\tau(s)}(v,v)|^2\r)\\
    &=-2f^{-2}(X_{\tau(s)})|W^k_{\tau(s)}(v,v)|^2\langle \nabla f(X_{\tau(s)}), \ptr_{\tau(s)}\,dB_s \rangle \\
    &\quad +(3|\nabla f|^2-f\Delta f)f^{-2}(X_{\tau(s)})|W^k_{\tau(s)}(v,v)|^2\,ds\\
    &\quad +f^{-2}(X_{\tau(s)})\,d|W^k_{\tau(s)}(v,v)|^2\\
    &\quad + 2f^{-1}(X_{\tau(s)})\,\langle R(\nabla f, Q_{\tau(s)}(k(\tau(s))v))Q_{\tau(s)}(v),W^k_{\tau(s)}(v,v)\rangle\,ds.
\end{align*}
According to the definition of $W^k_s$, we thus obtain
\begin{align*}
  d&\l(f^{-2}(X_{\tau(s)})|W^k_{\tau(s)}(v,v)|^2\r)\\
   &=-2f^{-2}(X_{\tau(s)})|W^k_{\tau(s)}(v,v)|^2\langle \nabla f(X_{\tau(s)}), \ptr_{\tau(s)} dB_s \rangle \\
   &\quad +2f^{-1}(X_{\tau(s)})\langle R(\ptr_{\tau(s)}\,dB_s,Q_{\tau(s)}(k(\tau(s))v))Q_{\tau(s)}(v), W^k_{\tau(s)}(v,v) \rangle \\
   &\quad +(3|\nabla f|^2-f\Delta f)(X_{\tau(s)})f^{-2}(X_{\tau(s)})|W^k_{\tau(s)}(v,v)|^2\,ds\\
   &\quad + 2f^{-1}(X_{\tau(s)})\langle R(\nabla f, Q_{\tau(s)}(k(\tau(s))v))Q_{\tau(s)}(v),W^k_{\tau(s)}(v,v)\rangle\,ds\\
   &\quad +|R^{\sharp, \sharp}(Q_{\tau(s)}(k(\tau(s))v),Q_{\tau(s)}v)|_{\text{\tiny HS}}^2\,ds\\
   &\quad -\langle ({\bf d}^*R+\nabla {\rm Ric}^\sharp)(Q_{\tau(s)}(k(\tau(s))v), Q_{\tau(s)}(v)), W^k_{\tau(s)}(v,v) \rangle\,ds\\
   &\quad -{\rm Ric}(W^k_{\tau(s)}(v,v),W^k_{\tau(s)}(v,v))\,ds.
\end{align*}
Since all geometric quantities are bounded on $D$ and
$|Q_{\tau(s)}|\leq \exp\big(-\frac{1}{2}K_0\tau(s)\big)$, we have
\begin{align*}
d&\l(f^{-2}(X_{\tau(s)})|W^k_{\tau(s)}(v,v)|^2\r)\\
&\overset{\text{\tiny m}}{\leq }  \big(3|\nabla f|^2-f\Delta f\big)(X_{\tau(s)})f^{-2}(X_{\tau(s)})|W^k_{\tau(s)}(v,v)|^2\,ds\\
&\quad+2|R|(X_{\tau(s)}) \e^{-K_0\tau(s)}|W^k_{\tau(s)}(v,v)|\, |\nabla f|(X_{\tau(s)})f^{-1}(X_{\tau(s)})\,ds+|R|^2(X_{\tau(s)})\e^{-2K_0\tau(s)}\,ds\\
&\quad+K_2\e^{-K_0\tau(s)}|W^k_{\tau(s)}(v,v)|\,ds-K_0|W^k_{\tau(s)}(v,v)|^2\,ds.
\end{align*}
For any $\delta_1>0$ and $\delta_2>0$, it thus follows that 
\begin{align*}
d &\left(f^{-2}(X_{\tau(s)})|W^k_{\tau(s)}(v,v)|^2\right)\\
&\overset{\text{\tiny m}}{\leq }  \l(3|\nabla f|^2-f\Delta f\r)(X_{\tau(s)})f^{-2}(X_{\tau(s)})|W^k_{\tau(s)}(v,v)|^2\,ds\\
&\quad +(\delta_1+1)K_1^2 \e^{-2K_0\tau(s)}\,ds+\frac{1}{\delta_1}|W^k_{\tau(s)}(v,v)|^2\, |\nabla f|^2(X_{\tau(s)})f^{-2}(X_{\tau(s)})\,ds\\
&\quad+\frac{\delta_2}{2}K_2^2\e^{-2K_0\tau(s)}\,ds+\l(\frac{1}{2\delta_2}-K_0\r)|W^k_{\tau(s)}(v,v)|^2\,ds\\
&=\l[\l(\frac{1}{2\delta_2}-K_0\r)f^2+\l(3+\frac{1}{\delta_1}\r)|\nabla f|^2-f\Delta f\r]f^{-2}(X_{\tau(s)})|W^k_{\tau(s)}(v,v)|^2\,ds\\
&\quad+\l((\delta_1+1)K_1^2+\frac{\delta_2}{2}K_2^2\r)\e^{-2K_0\tau(s)}\,ds\\
 &\leq c_1(f)f^{-2}(X_{\tau(s)})|W^k_{\tau(s)}(v,v)|^2\,ds+\l((\delta_1+1)K_1^2+\frac{\delta_2}{2}K_2^2\r)\e^{-2K_0\tau(s)}\,ds,
\end{align*}
where
\begin{align*}
c_1(f):=\sup_D\l\{\l(\frac{1}{2\delta_2}+K_0^-\r)f^2+\l(3+\frac{1}{\delta_1}\r)|\nabla f|^2-(f\Delta f)\r\}.
\end{align*}
This implies
\begin{align*}
&d\left(\e^{-c_1(f) s}f^{-2}(X_{\tau(s)})|W^k_{\tau(s)}(v,v)|^2 \right)\\
&\quad\overset{\text{\tiny m}}{\leq}  \l((\delta_1+1)K_1^2+\frac{\delta_2}{2}K_2^2\r) \e^{-c_1(f) s}\e^{-2K_0\tau(s)}\,ds\\
&\quad\leq \l((\delta_1+1)K_1^2+\frac{\delta_2}{2}K_2^2\r) \e^{(2K_0^--c_1(f)) s}\,ds.
\end{align*}
By integrating from $0$ to $s$ and taking expectations, we then arrive at
\begin{align}\label{esti-w1}
\Phi_1(s)=\E^x\l[f^{-2}(X_{\tau(s)})|W^k_{\tau(s)}(v,v)|^2\r]\leq \l((\delta_1+1)K_1^2+\frac{\delta_2}{2}K_2^2\r) \e^{c_1(f)s}\int_0^s\e^{(2K_0^--c_1(f)) r}\,dr.
\end{align}
On the other hand, concerning $\Phi_2$, we first observe
\begin{align*}
&d \l(f^{-2}(X_{\tau(s)})\e^{-K_0\tau(s)}\l(\int_0^{\tau(s)}\langle Q_u(\dot{k}(u)v),\ptr_ud B_u \rangle\r)^2\r)\\
&=d \bigg[f^{-2}(X_{\tau(s)})\e^{-K_0\tau(s)}\l(\int_0^{s}\dot{h}_1(u)f^{-1}(X_{\tau(u)})\langle Q_{\tau(u)}(v),\ptr_{\tau(u)}d B_u \rangle\r)^2\bigg]\\
&\mequal \e^{-K_0 \tau(s)}\dot{h}_1(s)^2 f^{-4}(X_{\tau(s)}) \,\langle Q_{\tau(s)}(v), Q_{\tau(s)}(v) \rangle\,ds\\
&\quad -K_0\e^{-K_0\tau(s)}\l(\int_0^{s}\dot{h}_1(u)f^{-1}(X_{\tau(u)})\langle Q_{\tau(u)} (v), \ptr_{\tau(u)}dB_{u} \rangle\r)^2\,ds\\
&\quad -4\e^{-K_0\tau(s)}f^{-3}(X_{\tau(s)})\int_0^s\dot{h}_1(u)f^{-1}(X_{\tau(u)})\,\langle Q_{\tau(u)}(v), \ptr_{\tau(u)}dB_u\rangle \,\dot{h}_1(s)\langle \nabla f(X_{\tau(s)}), Q_{\tau(s)}(v)  \rangle ds\\
&\quad +(3|\nabla f|^2-f\Delta f)(X_{\tau(s)})f^{-2}(X_{\tau(s)})\e^{-K_0\tau(s)} 
\l(\int_0^{s}\dot{h}_1(u)f^{-1}(X_{\tau(u)})\langle Q_{\tau(u)} (v), \ptr_{\tau(u)}dB_{u} \rangle\r)^2 ds\\
&\leq   \e^{-2K_0 \tau(s)}\dot{h}_1(s)^2 f^{-4}(X_{\tau(s)}) \,ds\\
&\quad -K_0\e^{-K_0\tau(s)}\l(\int_0^{s}\dot{h}_1(u)f^{-1}(X_{\tau(u)})\langle Q_{\tau(u)} (v), \ptr_{\tau(u)}dB_{u} \rangle\r)^2\,ds\\
&\quad +2 f^{-2}(X_{\tau(s)})|\nabla f|^2(X_{\tau(s)})\e^{-K_0\tau(s)}\l(\int_0^s\dot{h}_1(u)f^{-1}(X_{\tau(u)})\langle Q_{\tau(u)}(v), \ptr_{\tau(u)}dB_u\rangle\r)^2 ds\\
&\quad+2  \e^{-2K_0\tau(s)}\dot{h}_1(s)^2 f^{-4}(X_{\tau(s)}) ds\\
&\quad +(3|\nabla f|^2-f\Delta f)(X_{\tau(s)})f^{-2}(X_{\tau(s)})\e^{-K_0\tau(s)} 
\l(\int_0^{s}\dot{h}_1(u)f^{-1}(X_{\tau(u)})\langle Q_{\tau(u)} (v), \ptr_{\tau(u)}dB_{u} \rangle\r)^2 ds
\end{align*}
and conclude
\begin{align*}
 &d \l(f^{-2}(X_{\tau(s)})\e^{-K_0\tau(s)}\l(\int_0^{s}\dot{h}_1(u)f^{-1}(X_{\tau(u)})\langle Q_{\tau(u)}(v),\ptr_{\tau(u)}d B_u \rangle\r)^2\r)\\
&\overset{\text{\tiny m}}{\leq} 3\e^{-2K_0 \tau(s)}\dot{h}_1(s)^2 f^{-4}(X_{\tau(s)}) \,ds\\
&\quad +\l(-K_0f^2+5|\nabla f|^2-f\Delta f\r)(X_{\tau(s)})f^{-2}(X_{\tau(s)})\e^{-K_0\tau(s)} 
\bigg (\int_0^{s}\dot{h}_1(u)f^{-1}(X_{\tau(u)})\langle Q_{\tau(u)} (v), \ptr_{\tau(u)}dB_{u} \rangle\bigg)^2 ds\\
&\leq 3\e^{-2K_0 \tau(s)}\dot{h}_1(s)^2 f^{-4}(X_{\tau(s)}) \,ds\\
&\quad +c_2(f)f^{-2}(X_{\tau(s)})\e^{-K_0\tau(s)} 
\bigg (\int_0^{s}\dot{h}_1(u)f^{-1}(X_{\tau(u)})\langle Q_{\tau(u)} (v), \ptr_{\tau(u)}dB_{u} \rangle\bigg)^2 ds,
\end{align*}
where $c_2(f)$ is defined in \eqref{def-c2}.
Integrating the above inequality from $0$ to $s$ then yields 
\begin{align}\label{estimate-Q1}
&\E^x\l[f^{-2}(X_{\tau(s)})\e^{-K_0\tau(s)}\l(\int_0^{s}\dot{h}_1(u)f^{-1}(X_{\tau(u)})\langle Q_{\tau(u)}(v),\ptr_{\tau(u)}d B_u \rangle\r)^2\r]\notag\\
&\leq 3\e^{c_2(f)s}\int_0^s\e^{-c_2(f)u}\dot{h}_1(u)^2\E^x\l[\e^{-2K_0\tau(u)} f^{-4}(X_{\tau(u)})\r]\,du.
\end{align}
Hence it remains to estimate the term $\E^x\l[\e^{-2K_0\tau(u)} f^{-4}(X_{\tau(u)})\r]$. By It\^{o}'s formula,
we have
\begin{align*}
d&\l(\e^{-2K_0\tau(u)}f^{-4}(X_{\tau(u)})\r)\\
&\mequal-2K_0\e^{-2K_0\tau(u)}f^{-2}(X_{\tau(u)})\,du +(-2f\Delta f+10|\nabla f|^2)(X_{\tau(u)}) \e^{-2K_0 \tau(u)}f^{-4}(X_{\tau(u)})\,du\\
&\leq 2 c_2(f)\e^{-2K_0\tau(u)}f^{-4}(X_{\tau(u)})\,du
\end{align*}
which implies 
\begin{align*}
\E^x\l[\e^{-2K_0\tau(u)}f^{-4}(X_{\tau(u)})\r]\leq f^{-4}(x)\e^{2c_2(f)u}.
\end{align*}
Substituting this estimate into \eqref{estimate-Q1} leads to
\begin{align*}
\Phi_2(s)&=\E^x\l[f^{-2}(X_{\tau(s)})\e^{-K_0\tau(s)}\l(\int_0^{s}\dot{h}_1(u)f^{-1}(X_{\tau(u)})\langle Q_{\tau(u)}(v),\ptr_{\tau(u)}d B_u \rangle\r)^2\r]\\
&\leq 3f^{-4}(x)\e^{c_2(f)s}\int_0^s\dot{h}_1(u)^2\e^{-c_2(f)u}\e^{2c_2(f)u}\,du\\
&\leq 3f^{-4}(x)\e^{c_2(f)s}\int_0^s\dot{h}_1(u)^2 \e^{c_2(f)u}\,du.
\end{align*}
By means of this inequality we can estimate the second integral on the right-hand side of \eqref{eq1}
to obtain
\begin{align*}
&\int_0^t \dot{h}_1(s)^2 \E^x\bigg[f^{-2}(X_{\tau(s)})\e^{-K_0\tau(s)}\l(\int_0^{s}\dot{h}_1(u)f^{-1}(X_{\tau(u)})\langle Q_{\tau(u)}(v),\ptr_{\tau(u)}d B_u \rangle\r)^2\bigg]\,ds\\
&\leq3f^{-4}(x)\int_0^t\dot{h}_1(s)^2\e^{c_2(f)s}\int_0^s\dot{h}_1(u)^2 \e^{c_2(f)u}\,du ds\\
&=\frac{3}{2}f^{-4}(x) \l(\int_0^t\dot{h}_1(s)^2\e^{c_2(f)s}\,ds\r)^2.
\end{align*}
Combining this estimate and inequality \eqref{esti-w1}  with inequality \eqref{eq1} yields
\begin{align*}
&\big|(\Hess \,u)_x(v,v)\big|\\
& \leq \sqrt{\|u\|_D\,u(x)}\sqrt{\l((\delta_1+1)K_1^2+\frac{\delta_2}{2}K_2^2\r)\,\lim_{t\rightarrow \infty}\int_0^{t}\dot{h}_1(s)^2  \e^{c_1(f)s}\int_0^s\e^{(2K_0^--c_1(f)) r}\,dr\,ds}\\
&\quad +2f^{-2}(x)\sqrt{\|u\|_D\,u(x)}\sqrt{3\lim_{t\rightarrow \infty}\int_0^{t}\dot{h}_1(s)^2\e^{c_2(f)s}\int_0^s\dot{h}_1(u)^2 \e^{c_2(f)u}\,du\,ds}\\
&\leq \sqrt{\|u\|_D\,u(x)}\sqrt{\l((\delta_1+1)K_1^2+\frac{\delta_2}{2}K_2^2\r)\,\lim_{t\rightarrow \infty}\int_0^{t}\dot{h}_1(s)^2  \e^{c_1(f)s}\int_0^s\e^{(2K_0^--c_1(f)) r}\,dr\,ds}\\
&\quad +\sqrt{6}f^{-2}(x)\sqrt{\|u\|_D\,u(x)}\,\lim_{t\rightarrow \infty}\bigg(\int_0^{t}\dot{h}_1(s)^2\e^{c_2(f)s}\,ds\bigg).
\end{align*}
It now suffices to choose a suitable function $h_1$.
By Lemma \ref{lem:est}, we have 
\begin{align*}
\e^{c_1(f)s}\int_0^s\e^{(2K_0^--c_1(f)) r}\,dr\leq \e^{c_1s}\int_0^s\e^{(2K_0^--c_1) r}\,dr\leq \e^{(c_1+2K_0^-)s}\int_0^s\e^{-c_1 r}\,dr \leq \frac{\e^{(c_1+2K_0^-)s}}{c_1}.
\end{align*}
We define
\begin{align*}
h_1(s)=\frac{\int_0^s\e^{-(c_1+2K^-_0)r}\,dr}{\int_0^t\e^{-(c_1+2K^-_0)r}\,dr}.
\end{align*}
If $\delta_1\leq \frac{1}{2}$, then it is easy to see that 
$
 c_2(f)\leq c_2\leq c_1
$
and $c_1\geq K_0^-$, 
\begin{align*}
|(\Hess \,u)_x(v,v)|
&\leq \sqrt{\|u\|_D\,u(x)}\sqrt{\l((\delta_1+1)K_1^2+\frac{\delta_2}{2}K_2^2\r)\frac{c_1+2K_0^-}{c_1}}\\
&\quad +f^{-2}(x)\sqrt{6\|u\|_D\,u(x)}(c_1+2K_0^-)\\
&\leq \sqrt{\|u\|_D\,u(x)}\sqrt{(1+2{1}_{K_0\neq 0})\l((\delta_1+1)K_1^2+\frac{\delta_2}{2}K_2^2\r)}\\
&\quad +\sqrt{6\|u\|_D\,u(x)}\,\l((2\delta_2)^{-1}+3K_0^-+\frac{\pi \sqrt{(n-1)K^-_0}}{2 \delta_x} + \frac{\pi^2 (n+3+\delta_1^{-1})}{4\delta_x^2}\r)
\end{align*}
for $0<\delta_1\leq \frac{1}{2}$ and $\delta_2>0$. We complete the proof of inequality \eqref{EHF-1} by taking the infimum with respect to $\delta_1$ and $\delta_2$.
\end{proof}

\begin{proof}[Proof of Theorem \ref{them-1} (Part II)] 
  Theorem \ref{them-harmonic} allows to control $|\Hess\,u(v,v)|$ directly
  in terms of $\|u\|_D$ as follows: 
\begin{align}
|(\Hess\,u)_x(v,v)|&=\lim_{t\rightarrow \infty}|(\Hess \,P^D_tu)_x(v,v)|\notag\\
& \leq \|u\|_D\sqrt{\E^x\left[\lim_{t\rightarrow \infty}\int_0^{\tau(t)}|W_s^k(v,v)|^2\,\dot{k}(s)^2\,ds\right]}\notag\\
&\quad +2\|u\|_D\,\E^x\l[\lim_{t\rightarrow \infty}\int_0^{\tau(t)}\|Q_s\|^2\,\dot{k}(s)^2\,ds\r]\notag\\
& \leq \|u\|_D\sqrt{\E^x\left[\lim_{t\rightarrow \infty}\int_0^{\tau(t)}|W_s^k(v,v)|^2\,\dot{k}(s)^2\,ds\right]}\notag\\
&\quad +2\|u\|_D\,\E^x\l[\lim_{t\rightarrow \infty}\int_0^{\tau(t)}\e^{-K_0s}\dot{k}(s)^2\,ds\r]={\rm I}+{\rm II}.\label{Hess-WQ}
\end{align}
For the second term ${\rm II}$ we observe
\begin{align*}
\E^x\l[\int_0^{\tau(t)}\e^{-K_0s}\dot{k}(s)^2\,ds\r]=\int_0^{t}\dot{h}_1(s)^2\E^x\big[\e^{-K_0\tau(s)}f^{-2}(X_{\tau(s)})\big]\,ds;
\end{align*}
the first term has been dealt with in  \eqref{W-esti} above. We conclude from \eqref{Hess-WQ} that
\begin{align*}
\big|(\Hess \,u)_x(v,v)\big|
& \leq \|u\|_D\sqrt{\lim_{t\rightarrow \infty}\int_0^{t}\dot{h}_1(s)^2\,\E^x\l[|W^k_{\tau(s)}(v,v)|^2f^{-2}(X_{\tau(s)})\r]\,ds}\\
&\quad +2\|u\|_D\,\lim_{t\rightarrow \infty}\int_0^{t}\dot{h}_1(s)^2\,\E^x\big[\e^{-K_0\tau(s)}f^{-2}(X_{\tau(s)})\big]\,ds.
\end{align*}
Recall that by Lemma \ref{lem:est}, 
\begin{align*}
c_3(f)&=\sup_D\l\{K_0^-f^2+3|\nabla f|^2-(f\Delta f)\r\}\\
&\leq K_0^-+\frac{\pi \sqrt{(n-1)K^-_0}}{2 \delta_x} + \frac{\pi^2 (n+3)}{4\delta_x^2}=:c_3<c_1.
\end{align*}
Thus 
\begin{align*}
\E^x\big[f^{-2}(X_{\tau(s)})\e^{-K_0\tau(s)}\big]\leq f^{-2}(x)\e^{c_3(f)s}\leq \e^{c_3s},
\end{align*}
and hence
\begin{align}
\big|(\Hess \,u)_x(v,v)\big|
& \leq \sqrt{\l((\delta_1+1)K_1^2+\frac{\delta_2}{2}K_2^2\r)\l(\int_0^{t}\dot{h}_1(s)^2 \e^{c_1s}\int_0^s\e^{(2K_0^--c_1) r}\,dr\,ds\r)}\,\|u\|_D\notag\\
&\qquad +2\l(\int_0^{t}\dot{h}_1(s)^2 \e^{c_3s}\,ds\r)\|u\|_D.\label{Hess-1}
\end{align}
Define
\begin{align*}
h_1(s)=\frac{\int_0^s\e^{-(c_1+2K_0^-)r}\,dr}{\int_0^t\e^{-(c_1+ 2K_0^-)r}\,dr},
\end{align*}
so that
\begin{align*}
\dot{h}_1(s)=\frac{\e^{-(c_1+2K_0^-)s}}{\int_0^t\e^{-(c_1+ 2K_0^-)r}dr}.
\end{align*}
The estimate 
\begin{align*}
\e^{c_1s}\int_0^s\e^{(2K_0^--c_1) r}\,dr\leq \e^{(c_1+2K_0^-)s}\int_0^s\e^{-c_1 r}\,dr \leq \frac{\e^{(c_1+2K_0^-)s}}{c_1}
\end{align*}
is immediate. Since $ c_3(f)\leq c_3\leq c_1$, we have
\begin{align*}
\int_0^t \frac{\e^{-2(c_1+2K_0^-)s}\e^{c_3s}}{\l(\int_0^t\e^{-(c_1+ 2K_0^-)r}dr\r)^2} \,ds&\leq \int_0^t \frac{\e^{-(c_1+2K_0^-)s}}{\l(\int_0^t\e^{-(c_1+ 2K_0^-)r}dr\r)^2} \,ds\\
&\leq \frac{1}{\int_0^t\e^{-(c_1+ 2K_0^-)r}dr}= \frac{c_1+ 2K_0^-}{1-\e^{-(c_1+ 2K_0^-)t}}.
\end{align*}
Using these estimates to bound \eqref{Hess-1} from above and letting $t$ tend to $\infty$, we deduce
\begin{align*}
|(\Hess \,u)_x(v,v)|
& \leq \sqrt{\l((\delta_1+1)K_1^2+\frac{\delta_2}{2}K_2^2\r)\l( \frac{c_1+2K_0^-}{c_1}\r)}\,\|u\|_D+2\l(c_1+2K_0^-\r)\|u\|_D\\
& \leq \sqrt{(1+2\1_{\{K_0^-\neq 0\}})\l((\delta_1+1)K_1^2+\frac{\delta_2}{2}K_2^2\r)}\,\|u\|_D\\
&\quad +2\l((2\delta_2)^{-1}+3K_0^-+\frac{\pi \sqrt{(n-1)K^-_0}}{2 \delta_x} + \frac{\pi^2 (n+3+\delta_1^{-1})}{4\delta_x^2}\r){\|u\|_D}
\end{align*}
for $\delta_1>0$ and $\delta_2>0$. This completes the proof of inequality \eqref{EHF-2}.
\end{proof}

\section{Estimate of the Hessian of the semigroup}

We are now going to apply our results to give explicit local Hessian estimates
for the heat semigroup on a Riemannian manifold.

\begin{remark}
Local Hessian estimates for the semigroup have been derived in
\cite[Section 4.2]{Holger} by using the Hessian formula in \cite{APT}.
We can improve these results significantly by clarifying the coefficients.
\end{remark}

\begin{theorem}\label{th-esti-HS}
Let $x\in M$ and $D$ be a relatively compact open domain such that $x\in D$.
% Let $\tau$  be a  stopping time such that $0<\tau<\tau_D$. 
Then for  $f\in \mathcal{B}_b(M)$,
\begin{align*}
|\Hess P_Tf|(x)
\leq  \inf_{\delta>0}\l\{\l(\frac{T}{2}\sqrt{ K_1^2+\frac{K_2^2}{\delta }}+\frac{2}{ T}\r)\exp\l(T\l(K_0^-+ \frac{\delta}{2}+\frac{\pi \sqrt{(n-1)K^-_0}}{2 \delta_x} + \frac{\pi^2 (n+3)}{4\delta_x^2}\r)\r)\r\}\|f\|_{D}.
\end{align*}
\end{theorem}

\begin{proof}
By Theorem \ref{th4} and Cauchy's inequality, we have
\begin{align}
\big|(\Hess {P_Tf})_x(v,v)\big|&\leq  \|f\|_{D}\l(\E^x\l[\int_0^T|W^k_s(v,\dot{k}(s)v)|^2\,ds\r]\r)^{1/2}+2\|f\|_{D}\l(\E^x\int_0^T|Q_s(\dot{k}(s)v)|^2\, ds\r)\notag\\
&\leq  \|f\|_D\sqrt{\E^x\left[\int_0^{\tau(T)}|W^k_s(v,v)|^2\,\dot{k}(s)^2\,ds\right]}+2\|f\|_D\,\E^x\l[\int_0^{\tau(T)}|Q_s|^2\,\dot{k}(s)^2\, ds\r]={\rm I}+{\rm II}.\label{HessPTf}
\end{align}
Using a similar argument as in the proof of Theorem \ref{them-1} (Part II) and Lemma \ref{th3}, we obtain
\begin{align*}
\big|(&\Hess {P_Tf})_x(v,v)\big|\\
  &\leq  \|f\|_{D} \l(\frac{(2p-1) ^{p}K_1^{2p}}{2\delta^{p-1}}+\frac{K_2^{2p}}{2\delta^{2p-1} }\r)^{1/(2p)}\e^{(K_0^-+\frac{1}{2}\delta)T}T^{1/p}\,\l[\E^x\int_0^T|\dot{k}(t)|^{2q}\,dt\r]^{1/(2q)}\\
  &\quad +2\|f\|_{D}\,\E^x\l[\int_0^T\e^{-K_0s}\dot{k}(s)^2\, ds\r]\\
&\leq \|f\|_{D} \l(\frac{(2p-1) ^{p}K_1^{2p}}{2^2\delta^{p-1}}+\frac{K_2^{2p}}{2^2\delta^{2p-1} }\r)^{1/(2p)}\e^{\big(K_0^-+\frac{1}{2}\delta +\frac{\tilde{c}_q}{2q}\big)T}T^{3/(2p)-1/2}\\
&\quad +2\|f\|_{D}\e^{K_0^-T}\E^x\l[\int_0^T\dot{k}(s)^{2}\,ds\r]\\
&\leq  \|f\|_{D} \l(\frac{(2p-1) ^{p}K_1^{2p}}{2^2\delta^{p-1}}+\frac{K_2^{2p}}{2^2\delta^{2p-1} }\r)^{1/(2p)}\e^{\big(K_0^-+\frac{1}{2}\delta +\frac{\tilde{c}_1}{2}\big)T}T^{(3-p)/2p} +\|f\|_{D}\e^{(K_0^-+\tilde{c}_1)T}\frac{2}{T}
\end{align*}
for any $\delta>0$. 
We complete the proof by substituting $\tilde{c}_1$ defined in Lemma \ref{th3} into the above inequality and letting $p \rightarrow 1$.
\end{proof}

For global results (i.e.~the case $D=M$), if $M$ is compact, Theorem \ref{th4}
holds with $\tau\equiv\infty$ and the Cameron-Martin valued process $k$
can be chosen deterministic and linear, which leads immediately to straightforward estimates.
If the manifold is non-compact, $\Ric\geq K_0$, $\|R\|_{\infty}:=\sup_{x\in M}|R|(x)<\infty$ and
\begin{align*}
&K_2:=\sup\left\{|({\bf d}^* R+\nabla \Ric)^{\sharp}(v,v)|\colon v,w\in T_xM,\ x\in M,\ |v|=1,\, |w|=1\right\}<\infty,
\end{align*}
then Theorem \ref{th-esti-HS} still holds when replacing $D$ by $M$ and letting $\delta_x\rightarrow \infty$.
In the following, we use a different argument by choosing $k\in C^1([0,T])$ such that $k(0)=0$ and $k(T)=1$ and estimating $\E^x|W^k_s(v,v)|^2$.

\begin{corollary}\label{cor1}
Let $M$ be a non-compact manifold. Assume $\Ric\geq K_0$ and $K_1:=\|R\|_{\infty}<\infty$,
\begin{align*}
&K_2=\sup\left\{|({\bf d}^* R+\nabla \Ric)^{\sharp}(v,v)|\colon v,w\in T_xM,\ |v|=1,\, |w|=1\right\}<\infty.
\end{align*}
Then,
\begin{align}\label{Eq:EstHessPTf}
|\Hess P_Tf|
\leq \l(\l(\frac{K_1^2}{2}+\frac{K_2^2}{2\delta }\r)^{1/2}+\frac{2}{T}\r)\exp\l(\l(K_0^-+ \frac{1}{2}\delta\r)T\r)\|f\|_{\infty}
\end{align}
for  $\delta>0$.
Moreover, if the manifold $M$ is Ricci parallel, i.e. $\nabla \Ric =0$, then 
\begin{align}\label{Eq:EstRicciFlat}
|\Hess P_Tf|
\leq  \l(\frac{K_1^2}{\sqrt{2}}+\frac{2}{T}\r)\e^{K_0^-T}\|f\|_{\infty}.
\end{align}
\end{corollary}
\begin{proof}
From \eqref{HessPTf}, we know that for $k\in C^1([0,T])$ such that $k(0)=0$ and $k(T)=1$,
\begin{align*}
\big|(\Hess {P_Tf})_x(v,v)\big|&\leq  \|f\|_{\infty}\l(\int_0^{T}\E^x|W^k_s(v,v)|^2\,\dot{k}(s)^2\,ds\r)^{1/2}+2\|f\|_{\infty}\,\E^x\l[\int_0^{T}|Q_s|^2\,\dot{k}(s)^2\, ds\r].
\end{align*}
If we choose $k(s)=(T-s)/T$, then it follows from \eqref{Weq} that
\begin{align*}
\big|(\Hess {P_Tf})_x(v,v)\big|&\leq  \l(\frac{K_1^2}{2}+\frac{K_2^2}{2\delta}\r)^{1/2}\e^{(K_0^-+\frac{\delta}{2})T}\|f\|_{\infty}+\frac{2}{T^2}\,\E^x\l[\int_0^{T}|Q_s|^2\, ds\r]\|f\|_{\infty}\\
&\leq \l(\l(\frac{K_1^2}{2}+\frac{K_2^2}{2\delta }\r)^{1/2}+\frac{2}{T}\r)\exp\l(\l(K_0^-+ \frac{1}{2}\delta\r)T\r)\|f\|_{\infty}
\end{align*}
If the manifold is Ricci parallel, then ${\bf d}^* R+\nabla \Ric=0$
and $\delta$ is not needed in the estimate of $\E^x|W_t(v,v)|^2$, see
inequality \eqref{est-Ws}. Thus, in this case, estimate
\eqref{Eq:EstHessPTf} reduces to~\eqref{Eq:EstRicciFlat}.
\end{proof}

%\proof[Acknowledgments]

\bibliographystyle{amsplain}%

\bibliography{New-Hessian}

\end{document}